\documentclass[reqno, 12pt, a4letter]{amsart}
\usepackage{amsmath,amsxtra,amssymb,latexsym, amscd,amsthm,pb-diagram}

\usepackage{epsfig}
\usepackage{graphics}
\usepackage{color}
\usepackage{ulem}
\usepackage[all]{xy}
\usepackage{geometry}

\numberwithin{equation}{section}
\theoremstyle{plain}
\newtheorem{thm}{Theorem}[section]

\newtheorem{prop}[thm]{Proposition}
\newtheorem{lemma}[thm]{Lemma}

\newtheorem{claim}{Claim}

\theoremstyle{definition}
\newtheorem{dfn}[thm]{Definition}
\newtheorem{ex}[thm]{Example}
\theoremstyle{remark}
\newtheorem{rem}[thm]{Remark}

\newcommand{\C}{\mathbb C}
\newcommand{\N}{\mathbb N}

\newcommand{\R}{\mathbb R}

\newcommand{\ve}{\varepsilon}

\def\dim{\operatorname{dim}}

\def\ees{{\accent"5E e}\kern-.385em\raise.2ex\hbox{\char'23}\kern-.08em}
\def\EES{{\accent"5E E}\kern-.5em\raise.8ex\hbox{\char'23 }}
\def\ow{o\kern-.42em\raise.82ex\hbox{
\vrule width .12em height .0ex depth .075ex \kern-0.16em \char'56}\kern-.07em}
\def\OW{O\kern-.460em\raise1.36ex\hbox{
\vrule width .13em height .0ex depth .075ex \kern-0.16em \char'56}\kern-.07em}

\begin{document}
\title[Relative homotopy groups for polynomial maps]{Relative homotopy groups and Serre fibrations for polynomial maps}

\author[Masaharu Ishikawa]{Masaharu Ishikawa}
\address{Faculty of Economics, Keio University, 4-1-1, Hiyoshi, Kouhoku, Yokohama, Kanagawa 223-8521, Japan}
\email{ishikawa@keio.jp}

\author[Tat-Thang Nguyen]{Tat-Thang Nguyen}
\address{Institute of Mathematics, Vietnam Academy of Science and Technology, 18 Hoang Quoc Viet road, Cau Giay district, 10307 Hanoi, Vietnam}
\email{ntthang@math.ac.vn}

\thanks{The second author thanks Vietnam Institute for Advanced Study in Mathematics (VIASM)  for warm hospitality during his visit. The first author is supported by JSPS KAKENHI Grant numbers JP19K03499, JP23K03098, 
%JP17H06128, 
JP23H00081 and Keio University Academic Development Funds for Individual Research. This work is supported by JSPS-VAST Joint Research Program, Grant number JPJSBP120219602 (for JSPS) and QTJP01.02/21-23 (for VAST)}

\begin{abstract}
Let $f$ be a polynomial map from $\R^m$ to $\R^n$ with $m>n>0$ and $t_0$ be a regular value of $f$. 
For a small open ball $D_{t_0}$ centered at $t_0$, 
we show that the map $f:f^{-1}(D_{t_0})\to D_{t_0}$ is a Serre fibration if and only if $f$ is a Serre fibration over
a finite number of 
certain simple arcs starting at $t_0$.
We characterize the fibration $f:f^{-1}(D_{t_0})\to D_{t_0}$ by relative homotopy 
groups 
defined for these arcs
and use it to prove the assertion.
\end{abstract}

\maketitle

\section{Introduction}

Let $f:\R^m\to\R^n$ be a polynomial map from $\R^m$ to $\R^n$ with $m>n>0$.
A fiber $f^{-1}(t_0)$ of $f$ is said to be {\it atypical} if 
there does not exist an open neighborhood $U_{t_0}$ of $t_0\in\R^n$ such that $f$ restricted to $f^{-1}(U_{t_0})$ is a locally trivial fibration. 
The set of the images of atypical fibers is a set of measure zero~\cite{Tho69, Ver76},
which is called the {\it bifurcation set} of $f$.
It is known by H.V.~H\`{a} and D.T.~L\^{e} that a fiber of a polynomial function $f:\C^2\to\C$ of complex two-variables is atypical if and only if its Euler characteristic is different from that of a general fiber~\cite{HL85}.
This result is generalized by C. Joi\c{t}a and M. Tib\u{a}r in~\cite{JT18} for a polynomial map $f:X\to Y$ between non-singular affine varieties $X$ and $Y$ with $\dim X=\dim Y+1=m\geq 3$ 
under the condition that there is no vanishing component. Here a vanishing component is a connected component of a nearby fiber that vanishes when it goes to $f^{-1}(t_0)$.
They also studied atypical fibers of algebraic maps $f:X\to\R^{m-1}$ from real non-singular irreducible algebraic sets $X$ with $\dim X=m\geq 3$ and characterized them by the Euler characteristics and the betti numbers of them and their nearby fibers~\cite{JT17}.
There are several other studies for real polynomial maps with $1$-dimensional fibers~\cite{cp, tz,  INP19}.
In general case, no characterization in terms of homology groups is known.

In~\cite[Theorem A]{Mei02}, G. Meigniez proved that a surjective map is a Serre fibration if and only if it is
a homotopic submersion and all vanishing cycles and all emerging cycles are trivial.
Further, it is proved in~\cite[Corollary 14 (3)]{Mei02} that a surjective homotopic submersion is a Serre fibration 
if and only if it is a Serre fibration over any arc in the image.

In this paper, we study the fibration of a polynomial map as a Serre fibration.
The aim is to detect an atypical fiber in the sense of a Serre fibration. We call it a {\it homotopically atypical fiber}.
We will show that, to determine if $f^{-1}(t_0)$ is homotopically atypical, we only need to check the fibrations over certain short simple arcs starting at $t_0$.
The short arcs are given as follows. Let $\mathfrak S_f$ be a locally finite stratification of $\R^n$ such that,
for each stratum $S\in \mathfrak S_f$, $f:f^{-1}(S)\to S$ is a locally trivial fibration.
Note that the bifurcation set of $f$ is the union of strata with dimension less than $n$.
The existence of such a stratification is shown in~\cite[Theorem 3.4]{DK11}.
For a point $t_0\in\R^n$, we choose a small open ball $D_{t_0}$ in $\R^n$ centered at $t_0$
such that 
$S'\cup\{t_0\}$ is simply-connected for each connected component $S'$ of $S\cap D_{t_0}$ 
for $S\in \mathfrak S_f$.
Then, the set $\{S'\}$ constitutes a stratification of $D_{t_0}$, which we denote by $\mathfrak S_{f,D_{t_0}}$.
For each stratum $S'\in \mathfrak S_{f,D_{t_0}}$, we choose a simple arc starting at $t_0$ and lying on $S'\cup\{t_0\}$. We call it an {\it s-arc on $S'$ at $t_0$} and denote it by $\delta_{S'}$.

\begin{thm}\label{thm01}
Let $f:\R^m\to \R^n$ be a polynomial map with $m>n>0$ and $t_0\in \R^n$ be a regular value of $f$.
Then, the map $f:f^{-1}(D_{t_0})\to D_{t_0}$ is a Serre fibration if and only if $f$ is a Serre fibration over an s-arc $\delta_{S'}$ on $S'$ at $t_0$ for each stratum $S'\in\mathfrak S_{f,D_{t_0}}$.
\end{thm}

Since $t_0$ is a regular value, $f:f^{-1}(D_{t_0})\to D_{t_0}$ is surjective. Moreover, it is always a homotopic submersion because it is a polynomial map. Therefore, we only need to check if all vanishing cycles and all emerging cycles are trivial. The point of the above theorem is that we only need to check the fibrations over s-arcs at $t_0$ while checking the fibrations over all arcs on $D_{t_0}$ is required in~\cite[Corollary 14 (3)]{Mei02}.

To prove the above theorem, we study relation between the fibration $f:f^{-1}(D_{t_0})\to D_{t_0}$ and the relative homotopy sets
$\pi_i(f^{-1}(\delta_{S'}), f^{-1}(t_0), x_0)$, where $x_0$ is a point in $f^{-1}(t_0)$ and $i\in\N$, by observing vanishing and emerging cycles. 
Note that $\pi_i(f^{-1}(\delta_{S'}), f^{-1}(t_0), x_0)$ is a group if $i\geq 2$.
This set is said to be trivial if it consists of only one element.
The fibration $f:f^{-1}(D_{t_0})\to D_{t_0}$ is characterized by the relative homotopy 
sets
as follows.

\begin{thm}\label{thm11}
Let $f:\R^m\to \R^n$ be a polynomial map with $m>n>0$ and $t_0\in \R^n$ be a regular value of $f$.
The map $f:f^{-1}(D_{t_0})\to D_{t_0}$ is a Serre fibration if and only if $f$ has no vanishing component at $t_0$ and $\pi_i(f^{-1}(\delta_{S'}), f^{-1}(t_0), x_0)$ are trivial for any $i\in\N$, $x_0\in f^{-1}(t_0)$, and any $S'\in\mathfrak S_{f,D_{t_0}}$.
\end{thm}

Remark that $\pi_i(f^{-1}(\delta_{S'}), f^{-1}(t_0), x_0)$ for $S'\in\mathfrak S_{f,D_{t_0}}$ does not depend on the choice of an s-arc $\delta_{S'}$ on $S'\cup\{t_0\}$ since $S'\cup\{t_0\}$ is simply-connected and $f$ is a locally trivial fibration over $S'$.

It is worth noting that for a complex polynomial map $f:\C^m\to\C^{m-1}$, it is known that the fiber $f^{-1}(t_0)$ over a regular value $t_0\in\C^{m-1}$ of $f$ is not atypical if and only if  the inclusion of each fiber $f^{-1}(t)$ into $f^{-1}(D_{t_0})$ is a weak homotopy equivalence for all $t\in D_{t_0}$~\cite[Theorem 4.4]{Ngu13}.
The above theorem asserts that, instead of checking the weak homotopy equivalence for all $t\in D_{t_0}$, we only need to check weak homotopy equivalences for 
a finite number of s-arcs starting at $t_0$.

In Section~2, we introduce the stratification that we use in this paper and recall the definitions of vanishing and emerging cycles in~\cite{Mei02}.
In Section~3, we study vanishing and emerging cycles of $f:f^{-1}(D_{t_0})\to D_{t_0}$ over s-arcs.
Section~4 is devoted to the proofs of Theorems~\ref{thm01} and~\ref{thm11}.
In Section~5, we give an example of a fibered map given by polynomials that has an atypical fiber with the same topology as its nearby fiber but has non-trivial vanishing and emerging $1$-cycles.

\section{Preliminaries}

Throughout the paper, for a topological space $X$, 
$\text{Int\,}X$ represents the interior of $X$ and $\partial X$ represents the boundary of $X$.
An $i$-dimensional sphere $S^i$ is the boundary of an $i+1$-dimensional ball. In particular, $S^0$ consists of two points.

\subsection{Stratification of $D_{t_0}$}

The stratification used in this paper is defined as follows.
Let $Z_0,Z_1,\ldots,Z_d$ be closed semi-algebraic sets in $\R^n$ such that
$Z_{i-1}\subset Z_i$ for $i=1,\ldots,d$, $Z_d=\R^n$, and each difference $Z_{i}\setminus Z_{i-1}$ is a smooth submanifold of $\R^n$.
Let $\mathfrak S$ be the set of connected components of  $Z_{i}\setminus Z_{i-1}$ for $i=0,1,\ldots,d$,
where $Z_{-1}=\emptyset$. 
By the construction, $\R^n$ is a disjoint union of smooth submanifolds in $\mathfrak S$. The set $\mathfrak S$ is a {\it strafitication} of $\mathbb R^n$.

Let $f:\R^m\to\R^n$ be a polynomial map from $\R^m$ to $\R^n$ with $m>n>0$.
As mentioned in~\cite[Theorem 3.4]{DK11}, there exists a stratification $\mathfrak S_f$ of $\R^n$ such that, for each $S\in\mathfrak S_f$, the map $f:f^{-1}(S)\to S$ is a locally trivial fibration.
Let $t_0$ be a point in $\R^n$. We can choose a small open ball $D_{t_0}$ in $\R^n$ centered at $t_0$ such that $S'\cup\{t_0\}$ is simply-connected for each connected component $S'$ of $S\cap D_{t_0}$ for $S\in \mathfrak S_f$. Then, the set $\{S'\}$ constitutes a stratification of $D_{t_0}$, which we denote by $\mathfrak S_{f,D_{t_0}}$. Each element in  $\mathfrak S_{f,D_{t_0}}$ is called a {\it stratum} of $\mathfrak S_{f,D_{t_0}}$. Since $\mathfrak S_f$ is locally finite, $\mathfrak S_{f,D_{t_0}}$ is a finite set.

\subsection{Vanishing and emerging cycles}

In this subsection, we recall the definitions of vanishing cycles and emerging cycles in~\cite{Mei02}.

Let $E$ and $B$ be topological spaces.
A map $f:E\to B$ is called a {\it homotopic submersion} if every map $g:X\to E$ from a polytope $X$ 
satisfies that every germ-of-homotopy for $f\circ g$ lifts to a germ-of-homotopy for $g$.
Here a germ-of-homotopy means that two homotopies $H, H':X\times [0,1] \to E$ are the same germ if they coincide in a neighborhood of $X\times \{0\}$.
In this paper, $f:E\to B$ is always a polynomial map and it is a homotopic submersion.

Let $S^i$ denote the $i$-dimensional sphere, where $i\geq 0$.
For a fibred map $g:S^i\times [0,1]\to E$, define $g_s:S^i\to E$ by $g_s(x)=g(x,s)$.

\begin{dfn}
A {\it vanishing $i$-cycle} is a fibred map $g:S^i\times  [0,1]\to E$ 
such that $g_s$ is null-homotopic in $f^{-1}(s)$ for $s>0$.
A vanishing $i$-cycle is said to be {\it trivial} if $g_0$ is also null-homotopic in $f^{-1}(0)$.
\end{dfn}

\begin{dfn}
An {\it emerging $i$-cycle} is a fibred map $g:S^i\times (0,1]\to E$ such that the image $g_s(x_0)$ of the base point $x_0\in S^i$ has a limit for $s\to 0$. An emerging $i$-cycle is said to be {\it trivial}
if there exist $\varepsilon>0$ and a fibered map $g':S^i\times [0,\varepsilon)\to E$ such that, 
for each $0<s<\varepsilon$,
the base points coincide, that is $g(x_0,s)=g'(x_0,s)$, and $g_s$ and $g_s'$ are homotopic in $f^{-1}(s)$ relative to the base point.
\end{dfn}

Using these notions, Meigniez proved the following theorem.

\begin{thm}[Theorem A in~\cite{Mei02}]\label{thmA}
A surjective map is a Serre fibration if and only if it is a homotopic submersion and all vanishing and emerging cycles are trivial.
\end{thm}

\begin{rem}\label{rem24}
If $f:E\to B$ is a Serre fibration then it has no vanishing component. Cf.~\cite[Example 4]{Mei02}.
\end{rem}

\section{Fibrations over s-arcs}\label{sec3}

Let $f:\R^m\to\R^n$ be a polynomial map from $\R^m$ to $\R^n$ with $m>n>0$.
For each $S'\in \mathfrak S_{f,D_{t_0}}$ except the case $S'=\{t_0\}$, choose an embedding $\gamma:[0,1]\to S'\cup\{t_0\}$ of the unit interval $[0,1]$ such that $\gamma(0)=t_0$ and $\gamma((0,1])\subset S'$.
The image of $\gamma$ is a simple arc on $S'\cup\{t_0\}$ and we call it an {\it s-arc on $S'$ at $t_0$}.
For an s-arc $\delta_{S'}$ on $S'$ at $t_0$, we consider the relative homotopy set $\pi_i(f^{-1}(\delta_{S'}), f^{-1}(t_0), x_0)$ for $i\geq 1$, where $x_0$ is a point in $f^{-1}(t_0)$. For simplicity, we denote it as $\pi_i(f,  \delta_{S'}, x_0)$. 
Note that $\pi_i(f,  \delta_{S'}, x_0)$ is a group if $i\geq 2$.
If $f^{-1}(t_0)$ is connected, we denote it as $\pi_i(f, \delta_{S'})$.
A trivial element in $\pi_1(f,  \delta_{S'}, x_0)$ is the element corresponding to the constant map to $x_0$.
The set $\pi_1(f,  \delta_{S'}, x_0)$ is said to be {\it trivial} if it consists of only the trivial element.
The notation $\pi_1(f,  \delta_{S'}, x_0)=1$ means that $\pi_1(f,  \delta_{S'}, x_0)$ is trivial.

%In this section, we prove the following theorem.

The main theorem in this section is the following.

\begin{thm}\label{prop21}
Suppose that $f$ has no vanishing component at $t_0$. Let $\delta_{S'}$ be an s-arc on 
$S'\in\mathfrak S_{f,D_{t_0}}$
at $t_0$. Then, the following are equivalent:
\begin{itemize}
\item[(1)] $\pi_i(f, \delta_{S'}, x_0)=1$ for any $i\geq 1$ and $x_0\in f^{-1}(t_0)$.
\item[(2)] All vanishing and emerging cycles of $f|_{f^{-1}(\delta_{S'})}$ are trivial.
\item[(3)] $f|_{f^{-1}(\delta_{S'})}$ is a Serre fibration.
\end{itemize}
\end{thm}

%The equivalence of (2) and (3) follows from Theorem~\ref{thmA}. We prove two propositions in the next two subsections and give the proof of Proposition~\ref{prop21}.

Throughout this section, we use the notations $I^i=[0,1]^i$ 
for $i\geq 1$,  $I^0=\{0\}$, 
$J_i=\partial I^i \setminus (\text{Int\,}I^{i-1}\times\{0\})$ for $i\geq 2$, and $J_1=\{0\}$.

\subsection{An interpretation of relative homotopy groups}

Let $B_r$ denote the $m$-dimensional closed ball centered at the origin $0\in\R^m$ with radius $r>0$.

\begin{lemma}\label{lemma1}
There exists a sufficiently large real number $R$ such that
\begin{itemize}
\item[(1)] $f^{-1}(t_0)$ is transverse to $\partial B_r$ for all $r\geq R$, and
\item[(2)] $f^{-1}(t_0)\setminus B_r$ is diffeomorphic to $(f^{-1}(t_0)\cap\partial B_r) \times \R$ for all $r\geq R$.
\end{itemize}
\end{lemma}

\begin{proof}
The assertion~(1) is well-known. 
Let $\psi:f^{-1}(t_0)\to\R$ be a function given by $\psi(x)=\|x\|^2$. 
By the Curve Selection Lemma~\cite{Mil68},
it can be proved that $\psi$ has no critical point outside $B_R$ if $R$ is sufficiently large.
Then the assertion~(2) follows from Ehresmann's Fibration Theorem~\cite{Ehr47}.
\end{proof}

We fix $R$ in Lemma~\ref{lemma1}
and assume that an s-arc $\delta_{S'}=\gamma([0,1])$ is sufficiently short 
so
that
$f^{-1}(t)$ is transverse to $\partial B_R$ for any $t\in \gamma([0,1])$.
In particular, the restriction of $f$ to $f^{-1}(\gamma([0,1]))\cap B_R$ is a locally trivial fibration.
Note that the assumption of $\delta_{S'}$ being short is not essential in the argument below since $f$ is a locally trivial fibration over $S'$. 
For a subset $X$ of $[0,1]$, we denote 
\[
 E_X=f^{-1}(\gamma(X)) \quad\text{and} \quad
 E^R_X=f^{-1}(\gamma(X))\cap B_R.
\]
If $X$ is a point $x$ in $[0,1]$ then we denote $E_{\{x\}}$ by $E_x$ for simplicity.
Set $E^-_X=E_X\setminus (E_0\setminus B_R)$.

This subsection is devoted to the proof of the next proposition.
This proposition is important since it shows that the relative homotopy groups $\pi_i(E_{[0,1]},E_0, x_0)$ over the interval $[0,1]$ are determined only by information described on a nearby fiber.

\begin{prop}\label{lemma210}
$\pi_i(E_{[0,1]},E_0, x_0)\cong \pi_i(E_1, E_1^R, x_1^R)$
for 
$i\geq 1$,
where $x_0$ is a point on $E_0$ 
and $x_1^R$ is a point in the intersection of $E_1^R$ and the connected component of $E_{[0,1]}^R$ containing $x_0$.
\end{prop}

The proposition is proved by combining the following four lemmas.

\begin{lemma}\label{lemma26}
$\pi_i(E_{[0,1]}, E_0, x_0)\cong \pi_i(E^-_{[0,1]}, E^R_0, x_0^R)$ for all $i\geq 1$,
where $x_0$ is a point in $E_0$ and $x_0^R$ is a point in the intersection of $E_0^R$ and the connected component of $E_0$ containing $x_0$. 
\end{lemma}

Remark that $x_0^R\in E_0^R$ always exists by the condition~(2) in Lemma~\ref{lemma1}.

\begin{proof}
It is enough to prove the assertion for $x_0^R=x_0$. We have the following commutative diagram of exact sequences:
\[
\begin{CD}
\pi_i(E_0^R, x_0^R) @>>> \pi_i(E_{[0, 1]}^{-}, x_0^R) @>>> \pi_i(E_{[0, 1]}^{-}, E_0^R, x_0^R)
   @>>> \pi_{i-1}(E_0^R, x_0^R)  \\
   @V{i^R_{*}}VV @V{i_{*}}VV @VVV @V{i^R_{*}}VV \\
\pi_i(E_0, x_0^R) @>>> \pi_i(E_{[0, 1]}, x_0^R) @>>> \pi_i(E_{[0, 1]}, E_0, x_0^R) 
   @>>> \pi_{i-1}(E_0, x_0^R),
\end{CD}
\]
where two lines are exact and vertical arrows are induced from the inclusions. The condition (2) in Lemma \ref{lemma1} implies that all morphisms $i_{*}^R$ are isomorphic. In order to prove the assertion, we will prove that $i_{*}$ are isomorphic.
The map $i_*:\pi_0(E_{[0, 1]}^{-}, x_0^R)\to \pi_0(E_{[0, 1]}, x_0^R)$ is a bijection because 
$\pi_0(E_{[0, 1]}^{-}, x_0^R)$ and $\pi_0(E_{[0, 1]}, x_0^R)$ correspond to the sets of connected components of $E^-_{[0,1]}$ and $E_{[0,1]}$, respectively. Hence, it is enough to show the isomorphisms for $i\geq 1$.

We first prove the surjectivity. 
For $i\geq 1$, let $[h]$ be an element in $\pi_i(E_{[0, 1]}, x_0^R)$, where $h: I^i\to E_{[0, 1]}$ is a continuous map such that $h(\partial I^i)= x_0^R$. 
Choose $a>0$ such that 
$h(I^i)\cap E_0 \subset B_{R+a}$, and choose $b>0$ and $c>0$ small enough so that $E_s$ is transverse to $\partial B_r$ for all $R\leq r\leq R+a+c$ and all $s\in [0, b]$ and 
$h(I^i)\cap E_{[0,b]}\subset B_{R+a+c}$.
The existence of $b$ satisfying the transversality condition can be proved by the Curve Selection Lemma. 
If we choose $b>0$ sufficiently small further then we can find $c>0$ satisfying the above condition.
The restriction of $f$ to $E_{[0,b]}\cap B_{R+a+c}$ is a locally trivial fibration. 
In particular, $E_{[0,b]}\cap (B_{R+a+c}\setminus \text{Int\,}B_R)$ is diffeomorphic to $\partial E_0^R\times [R,R+a+c]\times [0,b]$. Let $(x', r, s)$ be the coordinates of  $\partial E_0^R\times [R,R+a+c]\times [0,b]$, where $x'\in \partial E_0^R$, $r\in [R, R+a+c]$ and $s\in [0,b]$. 

We construct a homotopy $\phi_\tau: (E_{[0,b]}\cap B_{R+a+c})\cup E_{[b,1]}\to E_{[0, 1]}$ with parameter $\tau \in [0,1]$ as follows:
If $x=(x',r,s)\in E_{[0, b]}\cap (B_{R+a+c}\setminus \text{Int\,}B_R)$ with $s\leq\frac{r-R}{a+c}b$, 
then 
\[
    \phi_\tau(x):=\left(x', r, 
%    \left(1-\frac{r-R}{a+c}\right) s+\frac{r-R}{a+c}((1-\tau)s+\tau b)
\left(1-\frac{r-R}{a+c}\tau\right)s+\frac{r-R}{a+c}b\tau
\right)
\]
and $\phi_\tau(x)=x$ otherwise. See Figure~\ref{fig3}.
One can check that the map 
\[
   \Phi: ((E_{[0,b]}\cap B_{R+a+c})\cup E_{[b,1]})\times [0,1]\to E_{[0, 1]}, \quad (x, \tau)\mapsto \phi_\tau(x)
\]
is continuous, $\Phi(x, 0)=x$, and 
$\Phi(x, \tau)\in E^{-}_{[0, 1]}$ for all $x$ and all $\tau>0$. 
Then $\phi_1\circ h$ gives an element in  $\pi_i(E_{[0, 1]}^{-}, x_0^R)$ which is equal to $[h]$ in $\pi_i(E_{[0, 1]}, x_0^R)$. 
Thus, $i_*$ is a surjection.

\begin{figure}[htbp]
\includegraphics[scale=0.8, bb=130 581 459 707]{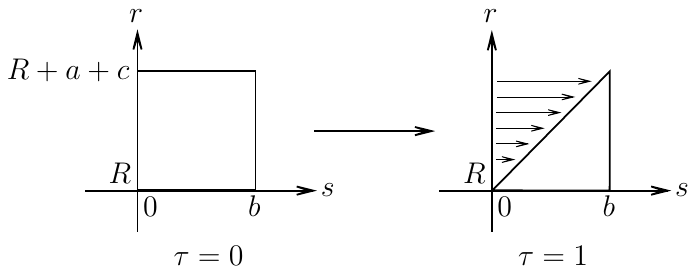}
\caption{The homotopy $\phi_\tau$.}
\label{fig3}
\end{figure}

Now we prove the injectivity. Let $[h_1], [h_2]$ be two elements in $\pi_i(E_{[0, 1]}^{-}, x_0^R)$, where $h_1, h_2: I^i\to E_{[0, 1]}^{-}$ such that there is a homotopy $H: I^i\times [0,1]\to E_{[0, 1]}$ 
with $H(\phantom{x}, 0)=h_1$, $H(\phantom{x}, 1)=h_2$ and $H(\partial I^i\times [0,1])= x_0^R$. 
Choose $a>0$ such that $B_{R+a}$ contains the image of $H$. Applying the same argument as above, we can construct a homotopy $\phi_\tau: (E_{[0, b]}\cap B_{R+a+c})\cup E_{[b,1]}\to E_{[0, 1]}$ such that $\phi_0(x)=x$ for all $x$, $\phi_1(x)\in E_{[0, 1]}^{-}$ for all $x$ and $\phi_\tau(x)=x$ for all $\tau$ and all $x\in B_R$. 
Thus $\phi_\tau\circ h_k$ gives a homotopy between $h_k$ and $\phi_1(h_k)$ in $E_{[0, 1]}^{-}$ for each $k=1, 2$. 
Also, $\phi_1\circ H(\phantom{x}, \tau)$ gives a homotopy between $\phi_1(h_1)$ and $\phi_1(h_2)$ 
in $E_{[0, 1]}^{-}$. Hence, $[h_1]= [h_2]$ in $\pi_i(E_{[0, 1]}^{-}, x_0^R)$. 
Thus, $i_*$ is an injection.
\end{proof}

\begin{lemma}\label{lemma27}
$\pi_i(E_{[0,1]}^-, E^R_0, x_0^R)\cong \pi_i(E_{[0,1]}^-, E^R_{[0,1]}, x^R_*)$ for
$i\geq 1$,
where $x_0^R$ is a point in $E_0^R$ and $x_*^R$ is a point
in 
%the intersection of $E_{[0,1]}^R$ and 
the connected component of 
$E_{[0,1]}^R$ 
containing $x_0^R$. 
\end{lemma}

Remark that since the restriction of $f$ to $E^R_{[0,1]}$ is a locally trivial fibration, there is one-to-one correspondence between the connected components of $E_0^R$ and those of $E_{[0,1]}^R$. 

\begin{proof}
Let $[h]$ be an element in $\pi_i(E_{[0,1]}^-, E^R_0, x_0^R)$, where $h:I^i\to E_{[0,1]}^-$ is a continuous map
such that $h(I^{i-1}\times \{0\})\subset E_0^R$ and $h(J_i)=x_0^R$.
Since $h(I^{i-1}\times \{0\})\subset E_0^R\subset E_{[0,1]}^R$, we may regard 
$[h]$ as an element in $\pi_i(E_{[0,1]}^-, E^R_{[0,1]}, x_0^R)$ and 
hence in $\pi_i(E_{[0,1]}^-, E^R_{[0,1]}, x_*^R)$.
Moreover, if $[h_1]=[h_2]\in\pi_i(E_{[0,1]}^-, E^R_0, x_0^R)$, then $h_1$ and $h_2$ are homotopic in $(E_{[0,1]}^-, E^R_{[0,1]}, x_*^R)$.
Thus we have a map
\[
   \Phi: \pi_i(E_{[0,1]}^-, E^R_0, x_0^R)\to \pi_i(E_{[0,1]}^-, E^R_{[0,1]}, x_*^R).
\]
We will prove that this map is bijective.

We first prove the surjectivity.
Let $[h']$ be an element in $\pi_i(E_{[0,1]}^-, E^R_{[0,1]}, x_*^R)$, where $h':I^i\to E_{[0,1]}^-$ is a continuous map such that $h'(I^{i-1}\times \{0\})\subset E_{[0,1]}^R$ and $h'(J_i)=x_*^R$.
We may replace $x_*^R$ by $x_0^R$ since they are in the same connected component of $E^R_{[0,1]}$. 
First, fix $\varepsilon>0$ sufficiently small so that 
$\partial B_{R+u}$ is transverse to $E_s$ for any $u\in [0,\varepsilon]$ and $s\in [0,1]$.
Then, $E_{[0,1]}\cap B_{R+\varepsilon}$ is diffeomorphic to $(E_0\cap B_{R+\varepsilon})\times [0,1]$. Let $(x,s)$ be the coordinates of $E_{[0,1]}\cap B_{R+\varepsilon}$,
where $x\in E_0\cap B_{R+\varepsilon}$ and $s\in [0,1]$.
Also, $E_{[0,1]}\cap (B_{R+\varepsilon}\setminus\text{Int\,}B_R)$ is diffeomorphic to 
$\partial E_0^R\times [0,\varepsilon]\times [0,1]$
and let $(x', u, s)$ be the coordinates of $E_{[0,1]}\cap (B_{R+\varepsilon}\setminus\text{Int\,}B_R)$,
where $x'\in \partial E_0^R$, $u\in [0,\varepsilon]$ and $s\in [0,1]$.
Note that $x=(x',u)$ on $E_{[0,1]}\cap (B_{R+\varepsilon}\setminus\text{Int\,}B_R)$.

Now we define an isotopy $\varphi_\tau$ in $E_{[0,1]}^-$, with parameter $\tau\in [0,1]$, by 
\[
\varphi_\tau(x,s)=
\begin{cases}
(x, (1-\tau)s) & (x,s)\in E_{[0,1]}^R, \\
(x', u, (1-\tau+\tau u/\varepsilon)s) & (x,s)\in E_{[0,1]}\cap (B_{R+\varepsilon}\setminus\text{Int\,}B_R),  \\
(x,s) & (x,s)\in E_{[0,1]}^-\setminus B_{R+\varepsilon}.
\end{cases}
\]
It satisfies that $\varphi_1(h'(I^i))\subset E_{[0,1]}^-$, $\varphi_1(h'(I^{i-1}\times\{0\}))\subset E_0^R$
and $\varphi_1(h'(J_i))=x_0^R$, that is, $[\varphi_1\circ h']\in \pi_i(E_{[0,1]}^-, E^R_0, x_0^R)$. 
The isotopy $\varphi_\tau$ shows that $\Phi([\varphi_1\circ h'])=[h']$. Thus $\Phi$ is surjective.

Next we prove the injectivity.
Let $[h_1]$ and $[h_2]$ be two elements in $\pi_i(E_{[0,1]}^-, E^R_0, x_0^R)$ such that $\Phi([h_1])=\Phi([h_2])$ in $\pi_i(E_{[0,1]}^-, E^R_{[0,1]}, x_0^R)$. 
Hence there exists a homotopy
$H:I^i\times [0,1]\to E_{[0,1]}^-$ such that $H(\phantom{x},0)=h_1$, $H(\phantom{x},1)=h_2$,
$H(I^{i-1}\times\{0\}, [0,1])\subset E_{[0,1]}^R$ and $H(J_i)=x_0^R$. 
Now we use the isotopy $\varphi_\tau$. Then 
$\varphi_1\circ H$ is a homotopy 
in $(E_{[0,1]}^-, E^R_0, x_0^R)$ such that  
$\varphi_1\circ H(\phantom{x},0)=h_1$ and 
$\varphi_1\circ H(\phantom{x},1)=h_2$. Hence $[h_1]=[h_2]$ in $\pi_i(E_{[0,1]}^-, E^R_0, x_0^R)$.
Thus $\Phi$ is injective.
\end{proof}

\begin{lemma}\label{lemma28}
$\pi_i(E_{[0,1]}^-, E^R_{[0,1]}, x_*^R)\cong \pi_i(E_{(0,1]}, E_{(0,1]}^R, x_*^R)$ 
for 
$i\geq 1$
and $x_*^R\in E_{(0,1]}^R$.
\end{lemma}

\begin{proof}
The inclusion $E_{(0,1]}\subset E_{[0,1]}^-$ induces a map $\Phi$ from $\pi_i(E_{(0,1]}, E_{(0,1]}^R, x_*^R)$ to $\pi_i(E_{[0,1]}^-, E^R_{[0,1]}, x_*^R)$.
Since $f$ is a 
submersion on $D_{t_0}$, 
we can make a vector field that pushes $E_{[0,1]}^-$ into $E_{(0,1]}$. 
Then, the argument similar to the previous lemma shows that $\Phi$ is bijective.
\end{proof}

\begin{lemma}\label{lemma29}
$\pi_i(E_{(0,1]}, E_{(0,1]}^R, x_1^R)\cong \pi_i(E_1, E_1^R, x_1^R)$ 
for 
$i\geq 1$
and $x_1^R\in E_1^R$.
\end{lemma}

\begin{proof}
Since the restriction of $f$ to $E_{(0,1]}$ is a trivial fibration and $\gamma((0,1])$ is contractible, we have the assertion. 
\end{proof}

Now Proposition~\ref{lemma210} follows from Lemmas~\ref{lemma26}, \ref{lemma27}, \ref{lemma28} and~\ref{lemma29}.

\subsection{Proof of Theorem~\ref{prop21}}

Before proving Theorem~\ref{prop21}, we show two lemmas.

\begin{lemma}\label{prop1}
Let $\delta_{S'}$ be an s-arc on 
$S'\in\mathfrak S_{f,D_{t_0}}$ at $t_0$ and $i\geq 0$.
If $f|_{f^{-1}(\delta_{S'})}$ has a non-trivial vanishing $i$-cycle, 
then $\pi_{i+1}(f, \delta_{S'}, x_0)\ne 1$,
where $x_0$ is a point in the connected component of $f^{-1}(t_0)$ intersecting the image of the vanishing $i$-cycle.
\end{lemma}

\begin{proof}
Let $\gamma:[0,1]\to D_{t_0}$ be an embedding of $[0,1]$ such that $\gamma([0,1])=\delta_{S'}$ and $\gamma(0)=t_0$. 
A non-trivial vanishing $i$-cycle is a continuous map $g:S^i\times [0,1]\to f^{-1}(\delta_{S'})$,
where $S^i$ is the $i$-dimensional sphere, such that $f(g(S^i\times\{s\}))=\gamma(s)$ and 
$g(S^i\times\{s\})$ is null-homotopic in $f^{-1}(\gamma(s))$ for $s\in (0,1]$ but $g(S^i\times\{0\})$ is not null-homotopic in $f^{-1}(t_0)$. In particular, there exists a continuous map $g':D^{i+1}\to f^{-1}(\gamma(1))$,
from a closed $i+1$-dimensional ball $D^{i+1}$, 
such that the restriction of $g'$ to $\partial D^{i+1}$ coincides with the restriction of $g$ to $S^i\times\{1\}$.
Then, there exists a continuous map $h:I^{i+1}\to f^{-1}(\gamma([0,1]))$ such that
$h(I^{i+1})=g(S^i\times [0,1])\cup g'(D^{i+1})$, 
$h(I^i\times\{0\})=g(S^i\times \{0\})\subset f^{-1}(t_0)$ 
and $h(J_{i+1})$ is a point $x_0$ on $g(S^i\times \{0\})$,
which is an element in $\pi_{i+1}(f, \delta_{S'}, x_0)$.

For contradiction, we assume that $[h]$ is trivial in $\pi_{i+1}(f, \delta_{S'}, x_0)$.
Then there exists a continuous map $h':I^{i+1}\to f^{-1}(t_0)$ such that
$h'(I^i\times\{0\})$ is homotopic to $h(I^i\times\{0\})$ in $f^{-1}(t_0)$ and $h'(J_{i+1})=x_0$.
Therefore, $h(I^i\times \{0\})$, 
which is $g(S^i\times \{0\})$, is null-homotopic in $f^{-1}(t_0)$,
which contradicts the assumption that the vanishing cycle $g$ is non-trivial.
Thus $[h]$ is non-trivial in $\pi_{i+1}(f, \delta_{S'}, x_0)$.
\end{proof}

\begin{lemma}\label{prop2}
Let $\delta_{S'}$ be an s-arc on 
$S'\in\mathfrak S_{f,D_{t_0}}$
at $t_0$ and $i\geq 1$.
If $f|_{f^{-1}(\delta_{S'})}$ has a non-trivial emerging $i$-cycle, 
then $\pi_i(f, \delta_{S'}, x_0)\ne 1$ for some $x_0\in f^{-1}(t_0)$.
\end{lemma}

\begin{proof}%[Proof of Proposition~\ref{prop2}]
We prove the contraposition. Assume that 
$\pi_i(f, \delta_{S'}, x_0)=\pi_i(E_{[0,1]}, E_0, x_0)=1$.
Let $g:S^i\times (0,1]\to E_{(0,1]}$ be an emerging $i$-cycle, where $S^i$ is the $i$-dimensional sphere.
For each $s\in (0,1]$, let $g_s:S^i\to E_s$ be the restriction of $g$ to $S^i\times\{s\}$.
By Proposition~\ref{lemma210}, we have 
$\pi_i(E_1, E_1^R, x_1^R)=1$, where $x_1^R$ is a point in the intersection of $E_1^R$ and the connected component of $E_{[0,1]}^R$ containing $x_0$.
Let $\phi: E_{(0,1]}\to  E_1$ be the projection induced by the diffeomorphism between $E_{(0,1]}$ and $E_1 \times (0,1]$.
For each $s\in (0,1]$, we define an isotopy $H_{s,\tau}$ of $\phi(g_s(S^i))$ by
$H_{s,\tau}=\phi(g_{(1-\tau)s+\tau}(S^i))$, where $\tau\in [0,1]$ is the parameter of the isotopy.
In particular, $H_{s,1}=g_1(S^i)$ for any $s\in (0,1]$. 
Since $\pi_i(E_1, E_1^R, x_1^R)=1$, we can push $g_1(S^i)$ into $E_1^R$ by a homotopy.
We denote the pushed $g_1(S^i)$ as $(S^i)^R$.

Recall that $f:E_{[0,1]}^R \to \gamma([0,1])$ is a locally trivial fibration. 
For the emerging $i$-cycle $g$, we set a continuous map $g':S^i\times [0,\varepsilon)\to E_{[0,1]}$
such that $f(g'(x,s))=\gamma(s)$ and $\phi(g'(S^i\times \{s\}))=(S^i)^R$ for $s\in [0,\varepsilon)$.
Regarding the homotopy from $g_1(S^i)$ to $(S^i)^R$ on $E_1$ as a homotopy on $E_s$ via the projection $\phi$, 
we can show that there exists a homotopy from $g_s(S^i)$ to $(S^i)^R\times \{s\}$ on $E_s$ 
for each $0<s<\varepsilon$.
Thus the emerging cycle $g$ is trivial. 
\end{proof}

\begin{proof}[Proof of Theorem~\ref{prop21}]
Since $f$ has no vanishing component at $t_0$, there is no non-trivial emerging $0$-cycle.
Then (1) $\Rightarrow$ (2) follows from  Lemmas~\ref{prop1} and~\ref{prop2}.
The implication (2) $\Rightarrow$ (3) follows from Theorem~\ref{thmA}.
Suppose that (3) holds. By~\cite[Corollary 14]{Mei02}, $f:(f^{-1}(\delta_{S'}), f^{-1}(t_0), x_0)\to (\delta_{S'},t_0, t_0)$ is a weak homotopy equivalence. Hence (1)  follows.
\end{proof}

\section{Proof of Theorem~\ref{thm11}}

Before proving  Theorem~\ref{thm11}, we show two lemmas.

\begin{lemma}\label{lemma42}
$f:f^{-1}(D_{t_0})\to D_{t_0}$ is a Serre fibration if and only if $f$ has no vanishing component at $t_0$ and $\pi_i(f, \delta', x')=1$ for any $i\in\N$, any s-arc $\delta'$ at any $t'\in D_{t_0}$ and any $x'\in f^{-1}(t')$.
\end{lemma}

\begin{proof}
The ``only if'' part follows from \cite[Corollary 14 (3)]{Mei02} and Theorem~\ref{prop21}. 
We prove the ``if'' part.
Suppose that $f:f^{-1}(D_{t_0})\to D_{t_0}$ is not a Serre fibration. 
By~\cite[Corollary 19]{Mei02}, 
there exists a straight line $\ell$ on $D_{t_0}$ such that $f|_{f^{-1}(\ell)}$ is not a Serre fibration. 
Since the stratification of the bifurcation set is semi-algebraic~\cite[Theorem 3.4]{DK11},
there exists an s-arc $\delta'$ on $\ell$ at some point $t'$ in $D_{t_0}$ such that $f|_{f^{-1}(\delta')}$ is not a Serre fibration.
Thus, by Remark~\ref{rem24} and
Theorem~\ref{prop21},
$f$ has no vanishing component and satisfies
$\pi_i(f, \delta', x')\ne 1$ for some $i\in\N$ and $x'\in f^{-1}(t')$.
\end{proof}

\begin{lemma}\label{lemma43}
Let $f:\R^m \to \R^n$ be a polynomial map
and $\gamma:[0,1]\to\R^n$ be an embedding of the interval $[0,1]$ into $\R^n$.
Suppose that 
$f:E_{(0,1]}\to (0,1]$ 
is a locally trivial fibration.
For $i\geq 0$, let
$g:S^i\times (0,1]\to E_{(0,1]}$ be a trivial emerging $i$-cycle, that is, there exists 
a fibered map $\bar g:S^i\times [0,\varepsilon]\to E_{[0,\varepsilon]}$ such that 
$g_s:=g(\phantom{x}, s):S^i\to E_s$ and $\bar g_s:=\bar g(\phantom{x},s):S^i\to E_s$ are homotopic in $E_s$
for each $s\in (0, \varepsilon]$, where $S^i$ is the $i$-dimensional sphere.
Then there exists a real number $\varepsilon'$ with $0<\varepsilon'<\varepsilon$ and 
a fibered map $\tilde g:S^i\times [0,1]\to E_{[0,1]}$ such that
$\tilde g_s=\bar g_s$ for $s\in [0,\varepsilon']$ and $\tilde g_s=g_s$ for $s\in [\varepsilon, 1]$,
where $\tilde g_s(x)=\tilde g(x, s)$.
\end{lemma}

\begin{proof}
Let $\varphi_\tau:S^i\to E_\varepsilon$ be the homotopy in $E_\varepsilon$ in the assertion with parameter $\tau\in [0,1]$ such that $\varphi_0=g_\varepsilon$ and $\varphi_1=\bar g_\varepsilon$. 
Since $f:E_{(0,1]}\to (0,1]$ is a locally trivial fibration, $E_{(0,1]}$ is diffeomorphic to $E_1\times (0,1]$.
Let $\phi:E_{(0,1]}\to E_1$ be the projection to the first factor of  $E_1\times (0,1]$.
Then we define $\tilde g:S^i\times (0,1]\to E_1\times (0,1]$ by 
\[
   \tilde g(x,s)=\begin{cases} 
   (\phi(\bar g_s(x)),s)  & s\in (0,\ve'] \\
   (\phi(\bar g_{2s-\varepsilon'}(x)),s) & s\in [\varepsilon', \frac{\varepsilon'+\varepsilon}{2}] \\
   (\phi(\varphi_{\frac{2s-\varepsilon'-\varepsilon}{\varepsilon-\varepsilon'}}(x)),s) & s\in [\frac{\varepsilon'+\varepsilon}{2}, \varepsilon] \\
   (\phi(g_s(x)),s) & s\in [\ve,1],
   \end{cases}
\]
where we identified $E_{(0,1]}$ with $E_1\times (0,1]$ by the diffeomorphism.
The map $\tilde g$ 
is a fibered map over $(0,1]$.
Since $\tilde g$ coincides with $\bar g$ on $S^i\times (0,\ve']$, it extends to a fibered map over $[0,1]$. This map satisfies the conditions in the assertion.
\end{proof}

\begin{proof}[Proof of Theorem~\ref{thm11}]
The ``only if'' part follows from \cite[Corollary 14 (3)]{Mei02} and Theorem~\ref{prop21}. 
We prove the ``if'' part.
Suppose that $f:f^{-1}(D_{t_0})\to D_{t_0}$ is not a Serre fibration
and has no vanishing component.
If there exists an s-arc $\delta_{S'}$ with $\pi_i(f, \delta_{S'}, x_0)\ne 1$ for some 
$i\in\N$ and $x_0\in f^{-1}(t_0)$, then the assertion follows.
We assume that there is no such an s-arc for a contradiction. 
By Lemma~\ref{lemma42}, there exists an s-arc $\delta'$ at some point $t'$ in $D_{t_0}\setminus\{t_0\}$
such that $\pi_i(f, \delta', x')\ne 1$ for some $i\in\N$ and $x'\in f^{-1}(t')$.
Let $\delta''$ be an s-arc connecting $t_0$ and $t'$ and lying on the closure of the stratum containing $t'$.

\begin{figure}[htbp]
\includegraphics[scale=0.9, bb=204 601 391 712]{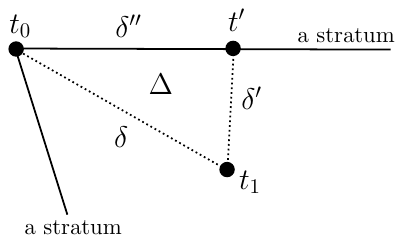}
\caption{A triangle $\Delta$ and s-arcs $\delta$, $\delta'$ and $\delta''$.}
\label{fig1}
\end{figure}

Let $g'$ be a vanishing $(i-1)$-cycle of  $f: f^{-1}(\delta')\to \delta'$, that is,
it is a fibered map
$g':S^{i-1}\times [0,1]\to f^{-1}(\gamma'([0,1]))$ such that 
$g'(S^{i-1}\times \{s\})$ is null-homotopic in $f^{-1}(\gamma'(s))$ for $s\in (0,1]$,
where $\gamma':[0,1]\to\delta'$ is a parametrization of $\delta'$ such that $\gamma'(0)=t'$.
Set $t_1=\gamma'(1)$. 
We choose a triangle $\Delta$ embedded in the closure of the stratum containing $t_1$ 
such that $\delta'$ and $\delta''$ are two of the three edges of $\Delta$,
and let $\delta$ denote the third edge of $\Delta$, see Figure~\ref{fig1}.
Suppose that $\delta''$ is parametrized as $\gamma''([0,1])=\delta''$ such that $\gamma''(0)=t_0$ and 
$\gamma''(1)=t'$. 
Then, 
since 
$f$ is a locally trivial fibration over $\delta''\setminus \{t_0\}$
and it has no vanishing component at $t_0$,
there exists a fibered map $g'':S^{i-1}\times (0,1]\to f^{-1}(\gamma''((0,1]))$ such that $g''(x, 1)=g'(x,0)$
and 
the limit of $g''(x_0,\varepsilon)$ for $\varepsilon\to 0$ exists, where 
$x_0$ is the base point of $S^{i-1}$.
The second condition is necessary to regard $g''$ as an emerging cycle, 
and is satisfied since $f$ has no vanishing component at $t_0$.

Since $\delta''$ is an s-arc at $t_0$, we have $\pi_i(f, \delta'', x_0)=1$ for any $i\in\N$
and $x_0\in f^{-1}(t_0)$ by the assumption in the first paragraph of this proof.
By Theorem~\ref{prop21} the emerging cycle $g''$ is trivial, and by Lemma~\ref{lemma43} there exists 
a fibered map $\tilde g'':S^{i-1}\times [0,1]\to f^{-1}(\gamma''([0,1]))$ such that
$\tilde g''(x, s)=g''(x, s)$ for $x\in S^{i-1}$ and $s\in (\varepsilon, 1]$,
where $\varepsilon>0$ is a sufficiently small real number.

Suppose that $\tilde g''(S^{i-1}, 0)$ is homotopic to a point $\bar x$ in $f^{-1}(t_0)$.
Concatenating $\tilde g''$ and this homotopy, we obtain a homotopy in $f^{-1}(\delta'')$ from $g''(S^{i-1},1)$ to $\bar x$.
Since $f$ over $\delta''=\gamma''([0,1])$ is a 
submersion
we can push this homotopy into $f^{-1}(\gamma''([\varepsilon_0,1]))$ for a sufficiently small real number $\varepsilon_0>0$, 
and further, since $f$ is a locally trivial fibration over $\gamma''([\varepsilon_0,1])$, we can push it into $f^{-1}(t')$.
Hence $g''(S^{i-1},1)=g'(S^{i-1},0)$ is null-homotopic in $f^{-1}(t')$,
which means that the vanishing $(i-1)$-cycle $g'$ is trivial.

Suppose that $\tilde g''(S^{i-1}, 0)$ is not null-homotopic in $f^{-1}(t_0)$.
Let $\hat g$ be the concatenation of $g'$ and $\tilde g''$. We push $\hat g(S^{i-1}\times [0,1])$ into 
$f^{-1}((\Delta\setminus\partial\Delta)\cup\{t_0\})$, 
fiberwisely, with fixing the points in $f^{-1}(t_0)$ and $f^{-1}(t_1)$. 
This is done by considering a lift of the vector field on $\Delta$ shown in Figure~\ref{fig2} that  brings $\delta'\cup\delta''$ to an s-arc $\delta_1$ at $t_0$ connecting $t_0$ and $t_1$ and lying on 
$(\Delta\setminus \partial\Delta)\cup\{t_0\}$. 
Then the lifted vector field brings $\hat g$ to 
a
fibered map $\hat g_1$ over $\delta_1$, fiberwisely.
Let $\gamma_1:[0,1]\to\delta_1$ be a parametrization of $\delta_1$ such that $\gamma_1(0)=t_0$ and 
$\gamma_1(1)=t_1$.
Since $\hat g_1(S^{i-1}, 1)$ is null-homotopic in $f^{-1}(t_1)$ and $f$ is a locally trivial fibration over 
$\Delta\setminus\partial\Delta$,
$\hat g_1(S^{i-1},u)$ is null-homotopic in $f^{-1}(\gamma_1(u))$ for any $u\in (0,1]$. Therefore, 
$\hat g_1$
is a vanishing $(i-1)$-cycle.
Since $\delta_1$ is an s-arc at $t_0$, we have $\pi_i(f, \delta_1, x_0)=1$ for any $i\geq 0$
and $x_0\in f^{-1}(t_0)$ by the assumption in the first paragraph 
of this proof.
Hence, by Theorem~\ref{prop21}, the vanishing cycle $\hat g_1$ is trivial. 
However, this contradicts the assumption that $\tilde g''(S^{i-1}, 0)$ is not null-homotopic in $f^{-1}(t_0)$.

\begin{figure}[htbp]
\includegraphics[scale=0.9, bb=204 601 391 712]{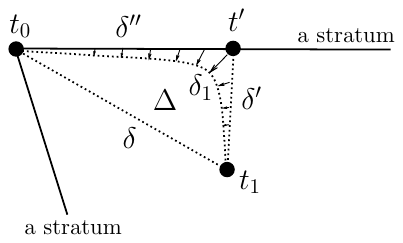}
\caption{A vector field on $\Delta$ whose lift brings $\hat g$ to $\hat g_1$.}
\label{fig2}
\end{figure}

Summarizing the arguments in the above two paragraphs, we conclude that the vanishing 
$(i-1)$-cycle $g'$ is trivial.

Next, let $h'$ be an emerging $i$-cycle of $f:f^{-1}(\delta')\to \delta'$. 
Since $f$ is a trivial fibration over 
$\Delta\setminus \partial\Delta$ 
and $f$ has no vanishing component at $t_0$,
we can obtain a fibered map
$h:S^i\times (0,1]\to f^{-1}(\gamma([0,1]))$, where $\gamma:[0,1]\to\delta$ is a parametization of $\delta$ with $\gamma(0)=t_0$ and $\gamma(1)=t_1$, such that 
\begin{itemize}
\item[(i')] $h$ is an emerging cycle of $f:f^{-1}(\delta)\to \delta$
and 
\item[(ii')] there exists an isotopy from $h(S^i\times (0,1])$ to $h'(S^i\times (0,1])$
associated with an isotopy of $\delta$ to $\delta'$ on $\Delta$
with keeping the endpoints on $t_1$ and $\delta''$.
\end{itemize}
Since $\delta$ is an s-arc of $f$ at $t_0$, it is assumed that $\pi_i(f, \delta, x_0)=1$.
Hence $h$ is a trivial emerging cycle by Theorem~\ref{prop21}, that is,
there exists a fibered map $\bar h:S^i\times [0,\varepsilon]\to f^{-1}(\gamma([0,\varepsilon]))$ such that
$h_s:=h(\phantom{x}, s):S^i\to f^{-1}(\gamma(s))$ and
$\bar h_s:=\bar h(\phantom{x}, s):S^i\to f^{-1}(\gamma(s))$ are homotopic in $f^{-1}(\gamma(s))$ for each $s\in (0, \varepsilon]$.
Consider a smooth vector field on $\Delta$ that induces the isotopy of $\delta$ to $\delta'$ in (ii') 
and make a non-zero smooth vector field in $f^{-1}(U_{\delta''})$ as a lift of this vector field,
where $U_{\delta''}$ is a small neighborhood of $\delta''$ in $\Delta$.
Pushing $\bar h$ by using this vector field, we obtain a fibered map $\bar h':S^i\times [0,\varepsilon]\to f^{-1}(\gamma'([0,\varepsilon]))$. 
For each $s\in (0,\varepsilon]$, $h'_s:=h'(\phantom{x},s)$ and $\bar h'_s:=\bar h'(\phantom{x}, s)$ are homotopic in $f^{-1}(\gamma'(s))$, which can be proved by pushing
into $f^{-1}(\gamma'(s))$ the concatenation of the inverse of the isotopy in (ii'), the homotopy in $f^{-1}(\gamma(s))$ 
and the isotopy induced by the above vector field. Thus the emerging cycle $h'$ is trivial.

Now it has been proved that any vanishing $(i-1)$-cycles and emerging $i$-cycles of $f: f^{-1}(\delta')\to \delta'$ are trivial.
Hence $\pi_i(f, \delta', x')=1$ by Theorem~\ref{prop21}, which is a contradiction.
\end{proof}

\begin{proof}[Proof of Theorem~\ref{thm01}]
If there is no vanishing component at $t_0$, then 
the assertion follows from Theorem~\ref{thm11} and Theorem~\ref{prop21}.
If there exists a vanishing component at $t_0$, then 
%$f:f^{-1}(D_{t_0})\to D_0$ is not a Serre fibration and 
there exists an s-arc over which $f$ is not a Serre fibration. Thus the assertion follows.
\end{proof}

\section{Examples}\label{sec5}

\begin{ex}[Joi\c{t}a and Tib\u{a}r~\cite{JT18}]
Let $\C^3\to\C^2$ be a polynomial map defined by
\[
f(x,y,z)=(x, ((x-1)(xz+y^2)+1)(x(xz+y^2)-1)).
\]
Over a small neighborhood of $(0,0)\in\C^2$, 
all fibers are diffeomorphic to $\C\sqcup\C$ but $f^{-1}(a,b)$ for $a\ne 0$ has a component that vanishes at infinity when $a$ goes to $0$.
Hence, there exists an s-arc $\delta$ starting at ${\bf 0}$ on which the fibered map has a non-trivial vanishing $0$-cycle and a non-trivial emerging $0$-cycle. 
In particular,  $\pi_1(f,\delta,x_0)\ne 1$ for $x_0\in f^{-1}({\bf 0})$ by Lemma~\ref{prop1}.
\end{ex}

\begin{ex}
We give an example of a fibered map, given as a restriction of a polynomial map $f:\R^5\to\R^3$, whose fibers are connected and have the same topology but that has a non-trivial vanishing $1$-cycle and a non-trivial emerging $1$-cycle. 
Let $f=(f_1,f_2,f_3)$ be a polynomial map from $\R^5$ to $\R^3$ given as
\[
\left\{
\begin{split}
   f_1(x,y,z,u,v)&=y^2+(u^2x+1)(vx-1)(x^2+(v-u^2)x+1) \\
   f_2(x,y,z,u,v)&=(z^2+u^2)-v(u^2+1)(z^2+1) \\
   f_3(x,y,z,u,v)&=u
\end{split}
\right.
\]
and set $\delta=\{(0,0,u)\in\R^3\mid u\in [0,1]\}$. 
In this section we will show that $f|_{f^{-1}(\delta)}:f^{-1}(\delta)\to \delta$ has no singularity, each fiber is diffeomorphic to $S^1\times \R$, but $f^{-1}(0,0,0))$ is an atypical fiber.

Set $g(x,u,v)=(u^2x+1)(vx-1)(x^2+(v-u^2)x+1)$. 
Regard $g$ as a polynomial of one variable $x$ and denote it as $g_{u,v}(x)$.
The condition $f_2=0$ implies $0\leq v<1$. Using this and $u\in [0,1]$, 
we can verify that the real roots of $g_{u,v}(x)=0$ are only $-\frac{1}{u^2}$ and $\frac{1}{v}$.
In particular, $g_{u,v}(x)=0$ has no multiple root.
This property will be used in the proofs below.

\begin{claim}
$f|_{f^{-1}(\delta)}$ has no singularity.
\end{claim}

\begin{proof}
The Jacobian matrix $J_f$ of $f$ is
\[
 J_f=\begin{pmatrix} \frac{\partial g}{\partial x} & 2y & 0 & \frac{\partial g}{\partial u} & \frac{\partial g}{\partial v} \\
 0 & 0 & 2z-2v(u^2+1)z & 2u-2uv(z^2+1) & -(u^2+1)(z^2+1) \\
 0 & 0 & 0 & 1 & 0
 \end{pmatrix}.
 \]
Hence a singular point of $f|_{f^{-1}(\delta)}$ satisfies $\frac{\partial g}{\partial x}=y=0$.
Then, $f_1=0$ implies $g=0$, but $g_{u,v}(x)=0$ cannot have a multiple root. Hence there is no singular point.
\end{proof}

\begin{claim}\label{claim2}
$f^{-1}(0,0,0)$ is diffeomorphic to $S^1\times\R$.
\end{claim}

\begin{proof}
Set $g_0(x,z)=g(x,0,\frac{z^2}{z^2+1})$. 
Then $f^{-1}(0,0,0)$ is identified with $X_0=\{(x,y,z)\in\R^3\mid y^2+g_0(x,z)=0\}$. Let $p:X_0\to\R$ be a smooth map defined by $p(x,y,z)=x-z^2$. 
Suppose that $p(x,y,z)$ is bounded. Then $x$ is bounded below. 
Assume that $x$ is not bounded above. 
If $vx-1>0$ 
then $g_0(x,z)=0$ has no solution. 
If $vx-1\leq 0$ and $x$ is sufficiently large then
$v=\frac{z^2}{z^2+1}\geq 0$ is close to $0$. Hence $z$ is also close to $0$.
This contradicts the assumption that $p(x,y,z)$ is bounded. Thus $p$ is proper.

A singular point of $p$ satisfies $g_0(x,z)=y=2z\frac{\partial g_0}{\partial x}+\frac{\partial g_0}{\partial z}=0$.
Since $g_0(x,z)=0$ and $y=0$, we have $x=\frac{1}{v}$. However
\[
2z\frac{\partial g_0}{\partial x}+\frac{\partial g_0}{\partial z}
%=-2zx\left(1+2v^2+\frac{x^2+2}{(z^2+1)^2}\right)\ne 0.
=\frac{2z(1+x^2)(x^2+2)}{x(z^2+1)^2}\ne 0.
\]
Thus $p$ is a submersion.

The map $p$ is a surjection since  $p(t, \sqrt{t^2+1}, 0)=t$. Hence it is a locally trivial fibration.
Since $x=z^2+1$ implies $vx=z^2$, $p^{-1}(1)$ is a circle given by $y^2+(z^2-1)(z^2+1)(z^2+2)=0$.
Hence $X_0$, which is $f^{-1}(0,0,0)$, is diffeomorphic to $S^1\times \R$.
\end{proof}

\begin{claim}
$f^{-1}(0,0,u)$ is diffeomorphic to $S^1\times \R$ for $0<u\leq 1$.
\end{claim}

\begin{proof}
Set $g_u(x,z)=g(x,u,\frac{z^2+u^2}{(u^2+1)(z^2+1)})$. 
Then $f^{-1}(0,0,u)$ is identified with $X_u=\{(x,y,z)\in\R^3\mid y^2+g_u(x,z)=0\}$. 
Let $q_u:X_u\to\R$ be a smooth map defined by $q_u(x,y,z)=z$.
Since $u\ne 0$, we have $0<v<1$. 
Suppose that $z$ is bounded. Then $v$ cannot be close to $0$ and $1$, and hence 
$g_u(x,z)$ is bounded below.
Then, $y^2+g_u(x,z)=0$ implies that $x$ and $y$ are bounded. Thus $q_u$ is proper.

A singular point of $q_u$ satisfies
$g_u(x,z)=y=\frac{\partial g_u}{\partial x}=0$.
However, $g_{u,v}(x)=0$ has no multiple root. 
Hence $q_u$ is a  submersion.

For each $z$, $v$ is determined by $v=\frac{z^2+u^2}{(u^2+1)(z^2+1)}$. 
Since $0<u\leq 1$, we have $v>0$. 
The point
$(x,y,z)=(\frac{1}{v},0,z)$ is 
in $q_u^{-1}(z)$,
which means that $q_u$ is a surjection, that is, it is a locally trivial fibration.
We can easily verify that $q_u^{-1}(0)=\{(x,y,0)\mid  y^2+g_u(x,0)=0\}$ is a circle. Thus $X_u$, which is $f^{-1}(0,0,u)$, is diffeomorphic to $S^1\times \R$.
\end{proof}

\begin{claim}\label{claim4}
The fibered map $f$ over $\delta$ has a non-trivial vanishing $1$-cycle and a non-trivial emerging $1$-cycle.
\end{claim}

\begin{proof}
Let $r:f^{-1}(\delta)\cap\{x-z^2=0\}\to \R$ be a map defined by $r(x,y,z,u,v)=u$.
Assume that there is a sequence on  $f^{-1}(\delta)\cap\{x-z^2=0\}$ with $|z|\to\infty$. 
Since $x\to\infty$ and $vx=\frac{(z^2+u^2)z^2}{(u^2+1)(z^2+1)}\to\infty$,
we see that $f_1=0$ has no solution. This means that there is no such a sequence.
If the $z$-coordinate of a sequence on $f^{-1}(\delta)\cap\{x-z^2=0\}$ is bounded, then 
$f_1=0$ implies that the $y$-coordinate is also bounded.
Hence the map $r$ is proper.

A singular point of $r$ satisfies $g(x,u,v)=y=0$, $f_2=0$ and $x-z^2=0$ and further satisfies either
\begin{itemize}
\item[(i)] $z=0$ or
\item[(ii)] $z\ne 0$ and 
$(u^2+1)(z^2+1)\frac{\partial g}{\partial x}+(1-v(u^2+1))\frac{\partial g}{\partial v}=0$.
\end{itemize}
The real roots of $g_{u,v}(x)=0$ are only $x=-\frac{1}{u^2}$ and $x=\frac{1}{v}$.
Thus, case (i) cannot happen.
In case (ii), since $x=-\frac{1}{u^2}$ contradicts $x-z^2=0$, we can assume that $x=\frac{1}{v}$.
However $(u^2+1)(z^2+1)\frac{\partial g}{\partial x}+(1-v(u^2+1))\frac{\partial g}{\partial v}=0$
implies $x^2+u^2+1=0$ and this has no solution. Hence $r$ is a submersion.
It has been proved in Claim~\ref{claim2} that $r^{-1}(0)$ is a circle. Hence the fiber of the proper submersion $r$ is a circle.

Next we show that the map $r_u:f^{-1}(0,0,u)\cap \{x-z^2\geq 0\}\to \R$ for $0<u\leq 1$ defined by $r_u(x,y,z)=x-z^2$ has only one singular point of Morse index $2$.
For each $u\in (0,1]$, the set of singular points of $r_u$ 
%satisfies 
is determined by
$g(x,u,v)=y=0$, $f_2=0$ and either (i) or (ii).
Since $x-z^2\geq 0$, the possible real root of $g_{u,v}(x)=0$ is $x=\frac{1}{v}$.
By $f_2=0$, we have
\begin{equation}\label{eq41}
   (u^2+1)(z^2+1)=x(z^2+u^2).
\end{equation}
In case~(i), 
$(x,y,z,u,v)=(\frac{u^2+1}{u^2},0,0,u,\frac{u^2}{u^2+1})$ is a singular point of $r_u$.
In case~(ii), 
we have $(1-x+z^2)(u^2+1)+x^2=0$,
and we
can confirm that this and equation~\eqref{eq41} have no common solution.

At the singular point $(\frac{u^2+1}{u^2},0,0,u,\frac{u^2}{u^2+1})$, the Hessian of $r_u$ is given as
\[
   \begin{pmatrix}  -2/\frac{\partial g}{\partial x} & 0 \\ 0 & -2 \end{pmatrix},
\]
where $\frac{\partial g}{\partial x}>0$. 
Hence
it is a Morse singularity of index $2$.
Then, we can conclude that, for each $0<u\leq 1$, the circle 
$r^{-1}_u(0)$ bounds a disk on $f^{-1}(0,0,u)$. 
On the other hand, the circle $r_0^{-1}(0)$ coincides with $p^{-1}(0)$,
where $p:X_0\to\R$ is the map in the proof of Claim~\ref{claim2}, 
and this circle is non-trivial in $\pi_1(X_0)=\pi_1(S^1\times \R)$ as seen in the proof. 
Hence the family of these circles is a non-trivial vanishing $1$-cycle over $\delta$.
If this fibered map does not have a non-trivial emerging $1$-cycle, then $f^{-1}(0,0,0)$ cannot be diffeomorphic to $S^1\times \R$. Hence it has a non-trivial emerging $1$-cycle.
\end{proof}

In particular, $\pi_2(f, \delta, x_0)$ and $\pi_1(f,\delta,x_0)$, $x_0\in f^{-1}(0,0,0)$,
are non-trivial by Lemmas~\ref{prop1} and~\ref{prop2}.
\end{ex}

\end{document}